\documentclass[12pt]{amsart}

\usepackage{stmaryrd}
\usepackage{amsmath}
\usepackage{amscd}
\usepackage{amssymb}
\usepackage{enumerate}
\usepackage{amsfonts}
\usepackage{graphicx}
\usepackage[all]{xy}
\usepackage{mathrsfs}
\usepackage{hyperref}

\usepackage{bbm}

\newcommand{\e}{\mathbbm{1}}

\newcommand{\limto}{{\displaystyle\lim_{\longrightarrow}}}
\newcommand{\rightlim}{\mathop{\limto}}

\newcommand{\leftlim}{\mathop{\displaystyle\lim_{\longleftarrow}}}
\newcommand{\limfromn}{\leftlim\limits_{\raise3pt\hbox{$n$}}}
\newcommand{\limton}{\rightlim\limits_{\raise3pt\hbox{$n$}}}

\newcommand{\rightlimit}[1]{\mathop{\lim\limits_{\longrightarrow}}\limits%
                   _{\raise3pt\hbox{$\scriptstyle #1$}}}
\newcommand{\leftlimit}[1]{\mathop{\lim\limits_{\longleftarrow}}\limits%
                   _{\raise3pt\hbox{$\scriptstyle #1$}}}

\numberwithin{equation}{section}

\newcommand{\rar}[1]{\stackrel{#1}{\longrightarrow}}

\newcommand{\into}{\hookrightarrow}

\newcommand{\al}{\alpha}
\newcommand{\be}{\beta}
\newcommand{\ga}{\gamma}

\newcommand{\de}{\delta}

\newcommand{\ze}{\zeta}

\newcommand{\eps}{\epsilon}
\newcommand{\sg}{\sigma}

\newcommand{\te}{\theta}

\newcommand{\Om}{\Omega}

\newcommand{\vp}{\varphi}

\newcommand{\bA}{{\mathbb A}}

\newcommand{\bF}{{\mathbb F}}
\newcommand{\bG}{{\mathbb G}}

\newcommand{\bP}{{\mathbb P}}
\newcommand{\bQ}{{\mathbb Q}}

\newcommand{\bZ}{{\mathbb Z}}

\newcommand{\cA}{{\mathcal A}}

\newcommand{\cG}{{\mathcal G}}

\newcommand{\cL}{{\mathcal L}}

\newcommand{\cO}{{\mathcal O}}

\newcommand{\cR}{{\mathcal R}}

\newcommand{\cW}{{\mathcal W}}

\newcommand{\sG}{{\mathscr G}}

\newcommand{\sX}{{\mathscr X}}

\newcommand{\Ker}{\operatorname{Ker}}

\newcommand{\End}{\operatorname{End}}
\newcommand{\Hom}{\operatorname{Hom}}

\newcommand{\Spec}{\operatorname{Spec}}

\newcommand{\pr}{\mathrm{pr}}
\newcommand{\ad}{\operatorname{ad}}

\newcommand{\Ind}{\operatorname{Ind}}

\newcommand{\tr}{\operatorname{tr}}
\newcommand{\Tr}{\operatorname{Tr}}

\newcommand{\Nrd}{\operatorname{Nrd}}
\newcommand{\Trd}{\operatorname{Trd}}

\newcommand{\tens}{\otimes}

\newcommand{\st}{\,\big\vert\,}

\newcommand{\sbr}{\smallbreak}
\newcommand{\mbr}{\medbreak}

\newtheorem{thm}{Theorem}[section]
\newtheorem{cor}[thm]{Corollary}
\newtheorem{lem}[thm]{Lemma}

\newtheorem{prop}[thm]{Proposition}

\theoremstyle{remark}
\newtheorem{rem}[thm]{Remark}
\newtheorem{rems}[thm]{Remarks}

\newtheorem{defin}[thm]{Definition}

\newcommand{\Fr}{\operatorname{Fr}}

\newcommand{\Gal}{\operatorname{Gal}}

\newcommand{\ql}{\overline{\bQ}_\ell}
\newcommand{\qls}{\overline{\bQ}_\ell^\times}

\newcommand{\bfq}{\overline{\bF}_q}
\newcommand{\fqn}{\bF_{q^n}}

\newcommand{\rec}{\operatorname{rec}}

\newcommand{\ceil}[1]{\lceil #1\rceil}

\begin{document}

\title[Special cases of LLC and JLC]{Geometric realization of special cases of local Langlands and Jacquet-Langlands correspondences}

\author[M.~Boyarchenko and J.~Weinstein]{Mitya Boyarchenko and Jared Weinstein}

\begin{abstract}
Let $F$ be a local non-Archimedean field, let $E\supset F$ be an unramified extension of degree $n\geq 2$ and let $\te$ be a smooth character of $E^\times$ such that $\te$ has level $r_0\geq 2$ and for each $1\neq\ga\in\Gal(E/F)$, the character $\te/\te^\ga$ has level $r_0$ as well (so $(E^\times,\te)$ is a minimal admissible pair in the terminology of $p$-adic representation theory). To $\te$ one associates a smooth irreducible $n$-dimensional representation $\sg_\te$ of the Weil group $\cW_F$ of $F$. It corresponds to an irreducible supercuspidal representation $\pi$ of $GL_n(F)$ via the local Langlands correspondence, which in turn corresponds to an irreducible representation $\rho$ of $D^\times$ via the Jacquet-Langlands correspondence, where $D$ is the central division algebra over $F$ with invariant $1/n$. In this note we give an explicit construction of $\pi$ and $\rho$. The result itself is not new: the proof that $\pi$ corresponds to $\sg_\te$ is a simple application of a more general result of Henniart, while the proof that $\rho$ corresponds to $\pi$ is almost identical to the proof of another result of Henniart, which was restricted to the case where $n$ is prime. However, our construction of $\pi$ and $\rho$ employs a new geometric ingredient (related to affinoid subspaces in the Lubin-Tate tower of $F$ found by the second author) that allows us to completely avoid the use of the Weil representation over finite fields that is required for the more algebraic approaches, and hence to simplify and streamline the key arguments. We also included some important details that already exist in other sources, in the hope that our text may be useful to those who are entering this research area for the first time.
\end{abstract}

\maketitle

\setcounter{tocdepth}{1}

\tableofcontents

\vfill\newpage

\section{Some results of Henniart and Kazhdan}

\subsection{Notation and terminology}\label{ss:notation} We use `LLC' and `JLC' as abbreviations for ``local Langlands correspondence for the group $GL_n$'' and ``(local) Jacquet-Langlands correspondence.'' We assume the existence and basic properties of the LLC and the JLC (notably, their compatibility with twists by $1$-dimensional representations).

\mbr

The following notation will be used throughout this text. We fix a nondiscrete locally compact non-Archimedean field $F$, a uniformizer $\varpi\in F$, an integer $n\geq 2$ and an unramified extension $E\supset F$ of degree $n$. The rings of integers of $F$ and $E$ will be denoted by $\cO_F$ and $\cO_E$, respectively. For any $m\in\bZ$, we write $P_E^m=\varpi^m\cdot\cO_E$; if $m\geq 1$, we also put $U^m_E=1+P_E^m\subset\cO_E^\times$. We let $\sG=\Gal(E/F)$.

\mbr

We denote the residue fields of $F$ and $E$ by $\bF_q$ and $\fqn$, respectively; whenever convenient we will identify $\sG$ with $\Gal(\fqn/\bF_q)$. We write $p=\operatorname{char}(\bF_q)$.

\mbr

Let $\vp\in\sG$ be the canonical generator, inducing the map $a\mapsto a^q$ on the residue field $\bF_{q^n}$. Consider the twisted polynomial ring $E\langle\Pi\rangle$ defined by the commutation relation $\Pi\cdot a=\vp(a)\cdot\Pi$ for all $a\in E$ and define $D=E\langle\Pi\rangle/(\Pi^n-\varpi)$. Then $D$ is a central division algebra over $F$ with invariant $1/n$. We also introduce the algebra $A=\End_F(E)\cong Mat_n(F)$ of all $F$-vector space endomorphisms of $E$, and we write $G=A^\times=GL_F(E)\cong GL_n(F)$. We view $E^\times$ as a subgroup of both $G$ and $D^\times$.

\mbr

We let $\cW_F$ and $\cW_E$ be the Weil groups of $E$ and $F$, respectively, and identify $\cW_E$ with an open normal subgroup of $\cW_F$ in the usual way (then $\cW_F/\cW_E$ is identified with $\sG$). We let $\rec_F:\cW_F\rar{}F^\times$ denote the continuous surjective homomorphism that induces the local class field theory isomorphism $\cW_F^{ab}\rar{\simeq}F^\times$, normalized in such a way that $\rec_F(\Phi)$ is a uniformizer in $F$ for any \emph{geometric} Frobenius element $\Phi\in\cW_F$. We also have the corresponding homomorphism $\rec_E:\cW_E\rar{}E^\times$.

\mbr

We fix a prime $\ell\neq p$ and an algebraic closure $\ql$ of $\bQ_\ell$. We will use $\ql$ as the coefficient field for all representations and adic sheaves considered in this text. In particular, by a \emph{character} of a topological group $H$ we will always mean a smooth homomorphism (i.e., a homomorphism with open kernel) $H\rar{}\qls$.

\mbr

If $\te$ is a character of $E^\times$, the \emph{level} of $\te$ is the smallest integer $r\geq 0$ such that $\te\bigl\lvert_{U^r_E}\equiv 1$, where $U^0_E=\cO_E^\times$ and $U^r_E$ is defined above for $r\geq 1$. The group $\sG$ acts on the set of all characters of $E^\times$: given $\ga\in\sG$, we write $\te^\ga(x)=\te(\ga(x))$. From now on we fix an integer $r_0\geq 2$. We will say that a character $\te$ of $E^\times$ of level $r_0$ is \emph{primitive} if $\te/\te^\ga$ has level $r_0$ for each $1\neq\ga\in\sG$. This implies that $\te$ has trivial stabilizer in $\sG$, so that $(E^\times,\te)$ is an admissible pair in Howe's terminology \cite{Howe-tame-supercuspidals-GL_n}.

\mbr

Finally, we fix a character $\eps$ of $F^\times$ whose kernel is equal to the image of the norm $N_{E/F}:E^\times\rar{}F^\times$. In particular, $\eps$ has order $n$ in the group $\Hom(F^\times,\qls)$.

\vfill\newpage

\subsection{Some special cases of the LLC}\label{ss:special-cases-LLC} Let $\sX$ denote the set of all characters of $E^\times$ that have trivial stabilizer in $\sG$. If $\te\in \sX$, then
\[
\sg_\te:=\Ind_{\cW_E}^{\cW_F} (\te\circ\rec_E)
\]
is a smooth irreducible $n$-dimensional representation of $\cW_F$. It is easy to prove

\begin{lem}\label{l:specify-Galois-side}
The map $\te\mapsto\sg_\te$ induces a bijection between the set of orbits $\sG\backslash\sX$ for the $\sG$-action on $\sX$ and the set $\cG_F^\eps(n)$ of isomorphism classes of smooth irreducible $n$-dimensional representations $\sg$ of $\cW_F$ that satisfy $\sg\cong\sg\tens(\eps\circ\rec_F)$.
\end{lem}

Let $\cA_F^\eps(n)$ denote the set of isomorphism classes of irreducible supercuspidal representations $\pi$ of the group $G=A^\times\cong GL_n(F)$ such that $\pi\cong\pi\tens(\eps\circ\det_A)$, where $\det_A:A^\times\rar{}F^\times$ is the usual determinant. Because the LLC is compatible with twists by characters of $F^\times$, it must restrict to a bijection between $\cG_F^\eps(n)$ and $\cA_F^\eps(n)$. In fact, the existence of a canonical bijection $\cG_F^\eps(n)\rar{\simeq}\cA_F^\eps(n)$ was known before the LLC was proved in general, and is a special case of a result of Kazhdan \cite{Kazhdan-on-lifting}; the results of \cite{Henniart-LLC-cyclic-Ann-Math} imply that the bijection found by Kazhdan is a restriction of the general LLC. Kazhdan's proof does not yield an \emph{explicit} construction of the bijection $\cG_F^\eps(n)\rar{\simeq}\cA_F^\eps(n)$; such a construction was found by Henniart in \cite{Henniart-MathNachr1992}. We now state the main result of \emph{op.~cit.} in the form in which we will use it.

\mbr

For each $\te\in\sX$, there is a purely algebraic construction of an irreducible supercuspidal representation $\pi_\te$ of $G$. Many special cases of this construction were found by Howe \cite{Howe-tame-supercuspidals-GL_n}; the general case is due to G\'erardin \cite{gerardin}, see also \cite{Henniart-MathNachr1992}.

\begin{rem}
Our usage of the notation $\pi_\te$ agrees with \cite{Henniart-JLC-I} but not with \cite{Henniart-MathNachr1992}.
\end{rem}

If $\pi$ is any smooth irreducible representation of $G$, its character (viewed as a generalized function on $G$) is represented by an ordinary locally constant function on the set of regular semisimple elements of $G$. By a slight abuse of notation, we denote this function simply by $g\mapsto\tr\pi(g)$.

\begin{thm}[Henniart]\label{t:Henniart-LLC}
Fix $\te\in\sX$.

\begin{enumerate}[$($a$)$]
\item The representations $\sg_\te\in\cG_F^\eps(n)$ and $\pi_{\xi\te}\in\cA_F^\eps(n)$ correspond to each other under the LLC, where $\xi$ is the character of $E^\times$ determined by $\xi(\varpi)=(-1)^{n-1}$ and $\xi\bigl\lvert_{\cO_E^\times}=1$ $($in particular, $\xi$ has order $1$ if $n$ is odd and $2$ if $n$ is even$)$.
 \sbr
\item There exists a constant $c=\pm 1$ such that
\[
\tr\pi_\te(x) = c\cdot\sum_{\ga\in\sG} \te^\ga(x)
\]
for every very regular element $x\in\cO_E^\times\subset G$.
\end{enumerate}
\end{thm}

Here, an element $x\in\cO_E^\times$ is called \emph{very regular} if the image of $x$ in $\fqn^\times=\cO_E^\times/U^1_E$ has trivial $\sG$-stabilizer. The constant $c$ is determined explicitly in terms of $\te$ in \cite[\S3.14]{Henniart-MathNachr1992}; we will not need the formula for $c$ in what follows.

\mbr

The proof of Theorem \ref{t:Henniart-LLC} consists of several references to the article \cite{Henniart-MathNachr1992}. We first introduce some notation and recall a result that will be useful for us later.

\mbr

Let $\cO_A=\End_{\cO_F}(\cO_E)\cong Mat_n(\cO_F)$ be the ring of $\cO_F$-module endomorphisms of $\cO_E$. We identify $\cO_A$ with an $\cO_F$-subalgebra of $A$, so that $\cO_A^\times$ becomes a maximal compact open subgroup of $G=A^\times$. The center of $G$ equals $F^\times\subset A^\times$.

\begin{lem}\label{l:conjugate-very-regular}
If $x\in\cO_E^\times$ is very regular and $g\in G$ is such that $gxg^{-1}\in F^\times\cdot\cO_A^\times$, then $g\in F^\times\cdot\cO_A^\times$.
\end{lem}

This lemma is well known and goes back to at least \cite{Carayol-representations-cuspidales}.

\begin{proof}
As $x$ is very regular, $\cO_E=\cO_F[x]$. Since $\det_A(gxg^{-1})=\det_A(x)\in\cO_F^\times$, we have $gxg^{-1}\in\cO_A^\times$. So the $\cO_F$-lattice $g^{-1}(\cO_E)\subset E$ is stable under $x$, and hence also under $\cO_F[x]=\cO_E$. Thus $g^{-1}(\cO_E)=P_E^m$ for some $m\in\bZ$, i.e., $g\in F^\times\cdot\cO_A^\times$.
\end{proof}

\begin{proof}[Proof of Theorem \ref{t:Henniart-LLC}]
Part (a) is equivalent to the Theorem stated in \cite[\S1.5]{Henniart-MathNachr1992}. To prove (b), recall from \emph{op.~cit.} that the construction of the representation $\pi_\te$ is such that $\pi_\te\cong\Ind_H^G(\sg)$ for some smooth irreducible representation $\sg$ of $H:=F^\times\cdot\cO_A^\times\subset G$ (it does not matter whether we use $\Ind$ or $c-\Ind$ in the last formula). The Frobenius character formula \cite[Thm.~A2]{Henniart-MathNachr1992} (see also \S\ref{s:Frobenius-character-formula} below) together with Lemma \ref{l:conjugate-very-regular} imply that $\tr\pi_\te(x)=\tr\sg(x)$ for each very regular element $x\in\cO_E^\times$. By the Theorem stated in \cite[\S3.14]{Henniart-MathNachr1992}, there exists a constant $c=\pm 1$ such that $\tr\sg(x) = c\cdot\sum_{\ga\in\sG} \te^\ga(x)$ for each such $x$, which completes the proof.
\end{proof}

\subsection{The strategy of our approach} In \S\ref{s:main-results} below we present a novel construction that to every primitive character $\te$ of $E^\times$ of level $r_0\geq 2$ associates an irreducible supercuspidal representation $\pi$ of $G$ and a smooth irreducible representation $\rho$ of $D^\times$. This construction involves a geometric ingredient that allows us to bypass the technical difficulties that arise in the more algebraic approaches. (To be more precise, our construction of $\pi$ (resp.~$\rho$) is only ``new'' in the case where $r_0$ is odd (resp. $r_0$ and $n$ are both even),
which is exactly when the more standard constructions rely on the Weil representation over a finite field.) We must then prove that $\pi$ and $\rho$ correspond to each other under the JLC, and that $\pi$ corresponds to $\sg_{\xi\te}$ under the LLC (where $\xi$ and $\sg_\te$ are defined in Theorem \ref{t:Henniart-LLC} and \S\ref{ss:special-cases-LLC}). A direct comparison with the constructions used in \cite{Henniart-MathNachr1992} and \cite{Henniart-JLC-I} is certainly possible, although it would involve exactly
the technical difficulties we wish to avoid. Therefore we prefer a more conceptual approach based on ideas we learned from \emph{op.~cit.}

\begin{prop}\label{p:Henniart-trick}
Let $\te$ be a character of $E^\times$ such that $\te/\te^\ga$ has level $\geq 2$ for each $1\neq\ga\in\sG$ $($in particular, $\te\in\sX${}$)$.

\begin{enumerate}[$($a$)$]
\item If $\pi$ is an irreducible supercuspidal representation of $G$ with central character $\te\bigl\lvert_{F^\times}$ such that $\pi\cong\pi\tens(\eps\circ\det_A)$ and such that there exists a constant $c'\neq 0$ satisfying $\tr\pi(x)=c'\cdot\sum_{\ga\in\sG} \te^\ga(x)$ for each very regular element $x\in\cO_E^\times$, then $\pi$ corresponds to $\sg_{\xi\te}$ under the LLC.
 \sbr
\item Let $\rho$ be a smooth irreducible representation of $D^\times$ with central character $\te\bigl\lvert_{F^\times}$ such that $\rho\cong\rho\tens(\eps\circ\Nrd_{D/F})$, where $\Nrd_{D/F}:D^\times\rar{}F^\times$ is the reduced norm, and such that there exists a constant $c''\neq 0$ satisfying $\tr\rho(x)=c''\cdot\sum_{\ga\in\sG} \te^\ga(x)$ for each very regular element $x\in\cO_E^\times$. If $\pi$ is the representation of $G$ satisfying the hypotheses of part (a), then $\pi$ corresponds to $\rho$ under the JLC.
\end{enumerate}
\end{prop}

The proof of Proposition \ref{p:Henniart-trick} is based on two lemmas.

\begin{lem}\label{l:supercuspidal-trick}
Let $\rho$ be a smooth irreducible representation of $D^\times$ and let $\pi$ be the representation of $G$ corresponding to $\rho$ under the JLC. If there exists a character $\eta$ of $F^\times$ of order $n$ such that $\rho\cong\rho\tens(\eta\circ\Nrd_{D/F})$, then $\pi$ is supercuspidal.
\end{lem}

\begin{proof}
Let $\sg$ be the $n$-dimensional Weil-Deligne representation of $\cW_F$ corresponding to $\pi$ under the LLC. Then $\sg$ is indecomposable. Assume, to obtain a contradiction, that $\sg$ is not irreducible. Let $\sg_0$ be the socle of $\sg$, i.e., the sum of all irreducible subrepresentations of $\sg$. Then $1\leq\dim(\sg_0)\leq n-1$. Since the JLC and the LLC are both compatible with twists, we obtain $\sg\cong\sg\tens(\eta\circ\rec_F)$, and therefore $\sg_0\cong\sg_0\tens(\eta\circ\rec_F)$. Taking determinants of both sides yields a contradiction.
\end{proof}

\begin{lem}[Henniart]\label{l:Henniart-trick}
Let $\te$ and $\te'$ be characters of $E^\times$. Assume that $\te\bigl\lvert_{F^\times}=\te'\bigl\lvert_{F^\times}$ and for each $1\neq\ga\in\sG$, the level of $\te/\te^\ga$ is $\geq 2$. If $c',c''\neq 0$ are constants with
\begin{equation}\label{e:equality-sums-theta-theta-prime}
c'\cdot\sum_{\ga\in\sG} \te'(\ga(x)) = c''\cdot\sum_{\ga\in\sG} \te(\ga(x))
\end{equation}
for every very regular element $x\in \cO_E^\times$, then $\te'=\te\circ\ga$ for some $\ga\in\sG$.
\end{lem}

\begin{proof}
We follow \cite[\S5.3]{Henniart-JLC-I}. The characters $\bigl\{\te\circ\ga\bigl\lvert_{U^1_E}\bigr\}_{\ga\in\sG}$ of $U^1_E$ are all pairwise distinct by assumption. If $x\in\cO_E^\times$ is very regular, so is $xy$ for any $y\in U^1_E$. Let us replace $x$ with $xy$ in formula \eqref{e:equality-sums-theta-theta-prime}, keep $x$ fixed and vary $y$ over $U^1_E$. We obtain an equation of linear dependence between the $2n$ characters of $U^1_E$ given by $\te\circ\ga$ and $\te'\circ\ga$ for all $\ga\in\sG$. Even though the characters $\te'\circ\ga$ of $U^1_E$ may not be \emph{a priori} pairwise distinct, we see that $c'=c''$ and there exists $\ga\in\sG$ such that $\te'\bigl\lvert_{U^1_E}=\te\circ\ga\bigl\lvert_{U^1_E}$. Without loss of generality we may assume that $\ga=1$. The same linear dependence argument now shows that $\te'(x)=\te(x)$ for any very regular element $x\in\cO_E^\times$. But if $x$ is very regular, then $x$ together with $U^1_E$ generate $\cO_E^\times$ as a group. Thus $\te$ and $\te'$ agree on $\cO_E^\times$ and hence also on $E^\times=F^\times\cdot\cO_E^\times$.
\end{proof}

\begin{proof}[Proof of Proposition \ref{p:Henniart-trick}]
(a) Since $\pi\in\cA_F^\eps(n)$, there exists $\te'\in\sX$ such that $\pi$ corresponds to $\sg_{\xi\te'}$ under the LLC. We must prove that $\te$ and $\te'$ are $\sG$-conjugate. We have $\pi\cong\pi_{\te'}$ by Theorem \ref{t:Henniart-LLC}(a), so by Theorem \ref{t:Henniart-LLC}(b), there exists a constant $c=\pm 1$ such that $c'\cdot\sum_{\ga\in\sG} \te'(\ga(x)) = c\cdot\sum_{\ga\in\sG} \te(\ga(x))$ for every very regular element $x\in \cO_E^\times$. The determinant of $\sg_{\xi\te'}$ equals $\bigl(\te'\bigl\lvert_{F^\times}\bigr)\circ\rec_F$, whence $\te\bigl\lvert_{F^\times}=\te'\bigl\lvert_{F^\times}$. Now Lemma \ref{l:Henniart-trick} implies that $\te$ and $\te'$ are $\sG$-conjugate.

\mbr

\noindent (b) Let $\pi''$ be the representation of $G$ corresponding to $\rho$ via the JLC. Then $\pi''$ is supercuspidal by Lemma \ref{l:supercuspidal-trick}, $\pi''\in\cA_F^\eps(n)$ because the JLC is compatible with twists, and $\tr\pi''(x)=(-1)^{n-1}\cdot\tr\rho(x)=(-1)^{n-1}c''\cdot\sum_{\ga\in\sG} \te^\ga(x)$ for each very regular element $x\in\cO_E^\times$. Furthermore, $\pi''$ has central character $\te\bigl\lvert_{F^\times}$ because the JLC preserves central characters. Part (a) of the proposition implies that $\pi''\cong\pi$.
\end{proof}

\subsection{An aside}
The assumption on $\te$ made in Proposition \ref{p:Henniart-trick} (which is equivalent to the requirement that $\te\bigl\lvert_{U^1_E}$ has trivial $\sG$-stabilizer) is only needed to apply Lemma \ref{l:Henniart-trick}. In fact, this assumption can be weakened substantially:

\begin{lem}\label{l:Henniart-trick-complicated}
Let $\te,\te'\in\sX$ be such that $\te\bigl\lvert_{F^\times}=\te'\bigl\lvert_{F^\times}$. In the case where $n=2$ and $q=3$, assume also that $\te\bigl\lvert_{U^1_E}$ is not $\sG$-invariant. If $c',c''\neq 0$ are constants such that \eqref{e:equality-sums-theta-theta-prime} holds for every very regular element $x\in \cO_E^\times$, then $\te'=\te\circ\ga$ for some $\ga\in\sG$.
\end{lem}

\begin{proof}
The argument of \cite[\S2.8]{Henniart-MathNachr1992} shows that if $\te\bigl\lvert_{U^1_E}$ is not $\sG$-invariant, then the conclusion of the lemma holds without any restrictions on $n$ and $q$. In addition, if $n>2$ or $n=2$ and $q>3$, the conclusion of the lemma follows from the argument given in \S\S2.6--2.11 of \emph{op.~cit.}
\end{proof}

\begin{rems}
\begin{enumerate}[(1)]
\item The irreducible representations of $G$ corresponding to those $\sg_\te$ that arise from characters $\te\in\sX$ such that $\te\bigl\lvert_{U^1_E}$ is $\sG$-invariant are exactly the twists of \emph{depth zero} supercuspidal representations. Indeed, if $\te\in\sX$, then $\te\bigl\lvert_{U^1_E}$ is $\sG$-invariant if and only if there exist a character $\te_1\in\sX$ of level $1$ and a character $\eta$ of $F^\times$ such that $\te=\te_1\cdot(\eta\circ N_{E/F})$.
 \sbr
\item Lemma \ref{l:Henniart-trick-complicated} is sharp, in the sense that the assumption that $\te\bigl\lvert_{U^1_E}$ is not $\sG$-invariant when $n=2$ and $q=3$ cannot be removed. To see this, we consider the case when $\te,\te'$ have level $1$, so that they can be viewed as characters of $\cO_E^\times/U^1_E\cong\bF_9^\times$. Let $y\in\bF_9^\times$ be a generator. The elements of $\bF_9^\times$ that have trivial $\sG$-stabilizer break up into the following three $\sG$-orbits: $\{y,y^3\}$, $\{y^{-1},y^{-3}\}$ and $\{y^2,y^{-2}\}$. To specify the characters $\te$ and $\te'$ we must specify $\ze=\te(y)$ and $\ze'=\te'(y)$; here $\ze,\ze'\in\qls$ must be $8$-th roots of unity that are not $\pm 1$. Note that $\te$ and $\te'$ are $\sG$-conjugate if and only if $\ze'\in\{\ze,\ze^3\}$. On the other hand, the hypothesis of the lemma with $c'=1$ and $c''=-1$ amounts to the following identities:
    \[
    \ze+\ze^3=-(\ze'+\ze'^3), \quad \ze^{-1}+\ze^{-3}=-(\ze'^{-1}+\ze'^{-3}), \quad \ze^2+\ze^{-2}=-(\ze'^2+\ze'^{-2}).
    \]
    If we choose $\ze$ to be a primitive $8$-th root of $1$ and take $\ze'=-\ze$, the identities above will be satisfied because $\ze^2+\ze^{-2}=0$; however, $\ze'\not\in\{\ze,\ze^3\}$. One can check by inspection that this is essentially the only counterexample; in particular, if we assume that $c'=c''$, then Lemma \ref{l:Henniart-trick-complicated} becomes valid without any additional hypotheses in the case where $n=2$ and $q=3$.
\end{enumerate}
\end{rems}

\section{Statements of the main results}\label{s:main-results}

From now on we fix a \emph{primitive} character $\te$ of $E^\times$ of level $r_0\geq 2$. Our goal in this section is to give a construction of an irreducible supercuspidal representation $\pi$ of $G$ that satisfies the hypotheses of Proposition \ref{p:Henniart-trick}(a) and of a smooth irreducible representation $\rho$ of $D^\times$ that satisfies the hypotheses of Proposition \ref{p:Henniart-trick}(b).

\subsection{Generalities} The representation $\pi$ will be constructed via induction from an open subgroup $K\subset G$ such that $F^\times\subset K$ and $K/F^\times$ is compact. Let us recall some standard facts that will allow us to prove the required properties of $\pi$.

\begin{defin}\label{d:intertwiners}
Let $G$ be an abstract group, let $K\subset G$ be a subgroup and let $\sg$ be a representation of $K$. For each $g\in G$, we write $\sg^g$ for the representation of $g^{-1}Kg$ given by $\sg^g(k)=\sg(gkg^{-1})$. One says that $g$ \emph{intertwines} the pair $(K,\sg)$ if $\Hom_{K\cap g^{-1}Kg}(\sg,\sg^g)\neq 0$, where $\sg,\sg^g$ are viewed as representations of $K\cap g^{-1}Kg$.
\end{defin}

The following result is well known.
\begin{thm}\label{t:supercuspidals-GL-n}
Let $K\subset G=A^\times\cong GL_n(F)$ be an open subgroup such that $F^\times\subset K$ and $K/F^\times$ is compact and let $\sg$ be a smooth irreducible representation of $K$. Assume that if $g\in G$ intertwines $(K,\sg)$, then $g\in K$. Then

\begin{enumerate}[$($a$)$]
\item the natural map $c-\Ind_K^G(\sg)\rar{}\Ind_K^G(\sg)$ is an isomorphism;
 \sbr
\item $\pi=\Ind_K^G(\sg)$ is an irreducible supercuspidal representation of $G$;
 \sbr
\item for each regular elliptic $x\in G$, we have
\begin{equation}\label{e:character-formula-GL-n}
\tr\pi(x) = \sum_{g\in K\backslash G,\ gxg^{-1}\in K} \tr\sg(gxg^{-1}).
\end{equation}
\end{enumerate}
\end{thm}

\begin{rems}
Recall that an element $x\in G$ is \emph{regular elliptic} if its characteristic polynomial is separable and irreducible over $F$. For example, very regular elements of $\cO_E^\times\subset G$ are regular elliptic. It is tacitly understood that the sum on the right hand side of \eqref{e:character-formula-GL-n} is finite when $x$ is regular elliptic; in \S\ref{s:Frobenius-character-formula} we give a proof of this formula following the appendix of \cite{Henniart-MathNachr1992}.
\end{rems}

Theorem \ref{t:supercuspidals-GL-n} has an analogue for the group $D^\times$ in place of $G$, which is essentially trivial because the quotient $D^\times/F^\times$ is compact (in particular, all smooth irreducible representations of $D^\times$ are finite dimensional).

\begin{prop}\label{p:division-algebra-auxiliary}
Let $K\subset D^\times$ be an open subgroup and let $\sg$ be a smooth irreducible representation of $K$. Assume that if $g\in D^\times$ intertwines $(K,\sg)$, then $g\in K$. Then $\rho=\Ind_K^{D^\times}(\sg)$ is a smooth irreducible representation of $G$ and\footnote{Note that we necessarily have $F^\times\subset K$, and hence $K$ has finite index in $D^\times$.}
\begin{equation}\label{e:character-formula-division-algebra}
\tr\rho(x) = \sum_{g\in K\backslash D^\times,\ gxg^{-1}\in K} \tr\sg(gxg^{-1}) \qquad\forall\,x\in D^\times.
\end{equation}
\end{prop}

\subsection{The $GL_n$ case} Recall that $E\supset F$ is an unramified degree $n$ extension, $A=\End_F(E)$ is the algebra of endomorphisms of $E$ as an $F$-vector space and $G=A^\times\cong GL_n(F)$. We have natural embeddings $E\into A$ (as an $F$-subalgebra) and $E^\times\into G$ and $\sG=\Gal(E/F)\into G$ as subgroups. They allow us to identify $A$ with the twisted group algebra of $\sG$ over $E$, that is, every element of $A$ can be written uniquely as $\sum_{\ga\in\sG}a_\ga\cdot\ga$ for some $a_\ga\in E$, and the following commutation relation holds: $\ga\cdot a=\ga(a)\cdot\ga$ for all $a\in E$ and all $\ga\in\sG$. Let $C\subset A$ denote the orthogonal complement of $E$ with respect to the trace pairing on $A$; then $C=\bigoplus_{1\neq\ga\in\sG} E\cdot\ga$.

\mbr

Recall that $\cO_A=\End_{\cO_F}(\cO_E)$, which is an open compact $\cO_F$-subalgebra of $A$. For each $m\in\bZ$ we put $P^m_A=\varpi^m\cdot\cO_A$ (where $\varpi$ is a uniformizer of $F$), and if $m\geq 1$, we define $U^m_G=1+P^m_A$. As $E$ is unramified over $F$, each of the subgroups $U^m_G$ is normalized by $E^\times$ and $P_A^m=P_E^m\oplus(C\cap P^m_A)$ for all $m\in\bZ$.

\begin{thm}\label{t:GL-n}
Let $\te$ be a primitive character of $E^\times$ of level $r_0\geq 2$.

\begin{enumerate}[$($a$)$]
\item Suppose that $r_0$ is even. Then there exists a unique character $\widetilde{\te}$ of $E^\times\cdot U^{r_0/2}_G$ that restricts to $\te$ on $E^\times$ and is trivial on $1+(C\cap P^{r_0/2}_A)$. The representation $\pi=\Ind_{E^\times\cdot U^{r_0/2}_G}^G(\widetilde{\te})$ of $G$ is irreducible and supercuspidal.
 \sbr
\item Suppose that $r_0$ is odd. There is an irreducible representation $\sg$ of $E^\times\cdot U^{(r_0-1)/2}_G$ such that $\tr\sg(x)=(-1)^{(n-1)}\cdot\te(x)$ for each very regular element $x\in\cO_E^\times$ and the restriction of $\sg$ to $F^\times\cdot U^1_E\cdot U^{(r_0+1)/2}_G$ is a direct sum of copies of a character that equals $\te$ on $F^\times\cdot U^1_E$ and is trivial on $1+(C\cap P^{(r_0+1)/2}_A)$. The representation $\pi=\Ind_{E^\times\cdot U^{(r_0-1)/2}_G}^G(\sg)$ of $G$ is irreducible and supercuspidal.
 \sbr
\item In both cases, $\pi$ has central character $\te\bigl\lvert_{F^\times}$ and satisfies $\pi\cong\pi\tens(\eps\circ\det_A)$, and $\tr\pi(x)=(-1)^{r_0(n-1)}\cdot\sum_{\ga\in\sG} \te^\ga(x)$ for each very regular element $x\in\cO_E^\times$.
\end{enumerate}
\end{thm}

This result, which by itself is not new, is proved in \S\ref{ss:proof-t:GL-n} below. The only new ingredient in our approach is the construction of the representation $\sg$ mentioned in part (b) of the theorem, which is obtained from an action of the group $E^\times\cdot U^{(r_0-1)/2}_G$ on a certain $(n-1)$-dimensional smooth affine hypersurface over $\fqn$ described in \S\ref{ss:geometric-ingredient} below. Similar remarks apply to Theorem \ref{t:division-algebra}, which we state next.

\subsection{The division algebra case} The central division algebra $D=E\langle\Pi\rangle/(\Pi^n-\varpi)$ over $F$ was constructed in \S\ref{ss:notation}. We have $D=\bigoplus_{j=0}^{n-1}E\cdot\Pi^j$ (a direct sum of left and right $E$-submodules). We write $\cO_D=\sum_{j=0}^{n-1}\cO_E\cdot\Pi^j$. For each $m\in\bZ$ we put $P_D^m=\Pi^m\cdot\cO_D$, and if $m\geq 1$, we let $U^m_D=1+P^m_D$. Each $U^m_D$ is a normal compact open subgroup of $D^\times$. If $C'=\sum_{j=1}^{n-1} E\cdot\Pi^j$, then $C'$ is equal to the orthogonal complement of $E$ with respect to the reduced trace pairing on $D$, and $P^m_D=(P^m_D\cap E)\oplus(P^m_D\cap C')$ for each $m\in\bZ$. Moreover, $P^m_D\cap E=P^{\ceil{m/n}}_E$ for all $m$, where $\ceil{a}$ denotes the smallest integer $\geq a$. The next result is proved in \S\ref{ss:proof-t:division-algebra}.

\begin{thm}\label{t:division-algebra}
Let $\te$ be a primitive character of $E^\times$ of level $r_0\geq 2$.

\begin{enumerate}[$($a$)$]
\item Suppose that $r_0$ is odd and let $m=\frac{n(r_0-1)}{2}+1$. There is a unique character $\widetilde{\te}$ of $E^\times\cdot U^m_D$ that restricts to $\te$ on $E^\times$ and is trivial on $1+(C'\cap P^m_D)$. The representation $\rho=\Ind_{E^\times\cdot U^m_D}^{D^\times}(\widetilde{\te})$ of $D^\times$ is irreducible.
 \sbr
\item Suppose $r_0$ is even and let $m=\frac{n(r_0-2)}{2}+1$. There is an irreducible representation $\sg$ of $E^\times\cdot U^m_D$ such that $\tr\sg(x)=(-1)^{(n-1)}\cdot\te(x)$ for each very regular element $x\in\cO_E^\times$ and the restriction of $\sg$ to $F^\times\cdot U^1_E\cdot U^{nr_0/2}_D$ is a direct sum of copies of a character that equals $\te$ on $F^\times\cdot U^1_E$ and is trivial on $1+(C'\cap P^{nr_0/2}_D)$. The representation $\rho=\Ind_{E^\times\cdot U^m_D}^{D^\times}(\sg)$ of $D^\times$ is irreducible.
 \sbr
\item In both cases, $\rho$ has central character $\te\bigl\lvert_{F^\times}$ and satisfies $\pi\cong\pi\tens(\eps\circ\Nrd_{D/F})$, and $\tr\pi(x)=(-1)^{(r_0-1)(n-1)}\cdot\sum_{\ga\in\sG} \te^\ga(x)$ for each very regular element $x\in\cO_E^\times$.
\end{enumerate}
\end{thm}

\subsection{A geometric ingredient}\label{ss:geometric-ingredient} Recall that $n,r_0\geq 2$ are integers and $q$ is a power of a prime number $p$. We introduce a (noncommutative) ring object $\cR_0$ in the category of affine $\bF_p$-schemes defined as follows. If $B$ is a commutative $\bF_p$-algebra, then $\cR_0(B)$ is the ring consisting of all formal expressions $a_0+a_1\cdot e_1+\dotsc+a_n\cdot e_n$, which are added in the obvious way and multiplied according to the following rules:

\begin{itemize}
\item $e_i\cdot a=a^{q^i}\cdot e_i$ for all $1\leq i\leq n$ and all $a\in B$
 \sbr
\item if $r_0=2$, then for all $i,j\geq 1$,
\[
e_i\cdot e_j =
\begin{cases}
e_{i+j} & \text{if } i+j\leq n, \\
0 & \text{otherwise}.
\end{cases}
\]
 \sbr
\item if $r_0>2$, then for all $i,j\geq 1$,
\[
e_i\cdot e_j =
\begin{cases}
e_n & \text{if } i+j=n, \\
0 & \text{otherwise}.
\end{cases}
\]
\end{itemize}

\begin{rems}
By construction, the additive group of $\cR_0$ is identified with $\bG_a^{n+1}$. If $r_0=2$, we see that we can identify $\cR_0(B)$ with the quotient $B\langle\tau\rangle/(\tau^{n+1})$, where the twisted polynomial ring $B\langle\tau\rangle$ is defined by the commutation relation $\tau\cdot a=a^q\cdot\tau$ for all $a\in B$. If $r_0>2$, then $\cR_0$ is independent of $r_0$.
\end{rems}

The multiplicative group $\cR_0^\times\subset\cR_0$ is given by $a_0\neq 0$. Let $U_0\subset\cR_0^\times$ be the subgroup defined by $a_0=1$. Then $U_0$ is a noncommutative $n$-dimensional connected unipotent algebraic group over $\bF_p$. We have a natural embedding $\bG_m\into\cR_0^\times$ that yields a semidirect product decomposition $\cR_0^\times=\bG_m\ltimes U_0$.

\mbr

We write $\cR^\times$ and $U$ for the algebraic groups over $\fqn$ obtained from $\cR^\times_0$ and $U_0$ by base change. Let $\Fr_{q^n}$ denote the $q^n$-power Frobenius morphism on $\cR_0$; for each commutative $\bF_p$-algebra $B$, the corresponding map $\Fr_{q^n}:\cR_0(B)\rar{}\cR_0(B)$ is given by $\sum_{i=0}^n a_i e_i \mapsto \sum_{i=0}^n a_i^{q^n}e_i$ (where we write $e_0=1$). We also denote by $\Fr_{q^n}$ the corresponding endomorphisms of the groups $\cR_0^\times,\cR^\times,U_0,U$. Finally, we write $L_{q^n}:U\rar{}U$ for the Lang isogeny, defined by $g\mapsto \Fr_{q^n}(g)g^{-1}$.

\begin{defin}
We set $X=L_{q^n}^{-1}(Y)$, where $Y\subset U$ is the hyperplane defined by $a_n=0$. The finite group $U(\fqn)$ acts on $X$ by right translation. In addition, the conjugation action of the subgroup $\fqn^\times\subset\cR^\times(\fqn)$ preserves $X$, and we thus obtain a right action of the group $\fqn^\times\ltimes U(\fqn)=\cR^\times(\fqn)$ on $X$. From now on we view $X$ as a variety over $\fqn$ equipped with this right action\footnote{Explicitly, the action is given by $x\bullet(\ga,u)=\ga^{-1}x\ga u$ for $x\in X$, $\ga\in\fqn^\times$, $u\in U(\fqn)$.} of the group $\cR^\times(\fqn)$.
\end{defin}

For any integer $j\geq 0$, the compactly supported cohomology $H^j_c(X\tens_{\fqn}\bfq,\ql)$ inherits a left action of $\cR^\times(\fqn)$ and becomes a finite dimensional representation of this finite group. Let us also observe that the center of the group $U(\fqn)$ is equal to $\{1+a_ne_n\st a_n\in\fqn\}$, so it can be naturally identified with $\fqn$, and it commutes with $\cR^\times(\fqn)$. In particular, if $\psi$ is any additive character of $\fqn$, the $\psi$-isotypic subspace $H^j_c(X\tens_{\fqn}\bfq,\ql)[\psi]$ (i.e., the subspace of the cohomology on which the center of $U(\fqn)$ acts via $\psi$) is an $\cR^\times(\fqn)$-subrepresentation of $H^j_c(X\tens_{\fqn}\bfq,\ql)$.

\begin{thm}\label{t:main-geometric-theorem}
Let $\psi:\fqn\rar{}\qls$ be a character with trivial $\Gal(\fqn/\bF_q)$-stabilizer.

\begin{enumerate}[$($a$)$]
\item $H^{n-1}_c(X\tens_{\fqn}\bfq,\ql)[\psi]$ is irreducible as a representation of $U(\fqn)$.
 \sbr
\item $H^j_c(X\tens_{\fqn}\bfq,\ql)[\psi]=0$ for all $j\neq n-1$.
 \sbr
\item If $\ze\in\fqn^\times\subset\cR^\times(\fqn)$ is any element that has trivial $\Gal(\fqn/\bF_q)$-stabilizer, then the trace of $\ze$ on $H^{n-1}_c(X\tens_{\fqn}\bfq,\ql)[\psi]$ equals $(-1)^{n-1}$.
\end{enumerate}
\end{thm}

This result is proved in \S\ref{s:proof-t:main-geometric-theorem} below. It is contained in one of the main results of \cite{maximal-varieties-LLC}, but the argument we give in \S\ref{s:proof-t:main-geometric-theorem} is easier to follow because the more general result proved in \emph{op.~cit.} involves a higher number of substantial ingredients.

\begin{rems}
\begin{enumerate}[(1)]
\item It follows from Lemma \ref{l:exercise2} that $H^{n-1}_c(X\tens_{\fqn}\bfq,\ql)[\psi]$ is the \emph{unique} irreducible representation of $U(\fqn)$ with central character $\psi$. It is also not hard to obtain an explicit realization thereof as the representation of $U(\fqn)$ induced from a $1$-dimensional representation of a suitable subgroup. On the other hand, the action of $\fqn^\times\subset\cR^\times(\fqn)$ is more subtle when $n$ is even.
 \sbr
\item The harder statement, proved in \cite{maximal-varieties-LLC}, is that if $\psi:\fqn\rar{}\qls$ is an \emph{arbitrary} character, then there exists an integer $n-1\leq k\leq 2n-2$ (depending on the stabilizer of $\psi$ in $\Gal(\fqn/\bF_q)$) such that $H^{j}_c(X\tens_{\fqn}\bfq,\ql)[\psi]$ vanishes for $j\neq k$, and is irreducible as a representation of $U(\fqn)$ for $j=k$.
 \sbr
\item Since $X$ is defined over $\fqn$, the Frobenius $\vp_{q^n}\in\Gal(\bfq/\fqn)$ acts on each space $H^{j}_c(X\tens_{\fqn}\bfq,\ql)$. It is proved in \emph{op.~cit.} that it acts by the scalar $(-1)^{j} q^{-nj/2}$.
 \sbr
\item As we will see, part (c) of the theorem is an easy consequence of part (b) and the fixed point formula of Deligne and Lusztig \cite{deligne-lusztig}.
\end{enumerate}
\end{rems}

\section{Frobenius character formula}\label{s:Frobenius-character-formula}

This section is independent of the rest of the text. Following the appendix of \cite{Henniart-MathNachr1992}, we sketch a proof of the fact that the Frobenius character formula is valid for certain induced representations of totally disconnected groups. Many different versions of this result are available in the literature, going back at least to \cite{Sally-some-remarks,Kutzko-supercuspidal-character-formulas} (which rely on earlier integral formulas of Harish-Chandra). We prefer the approach of \cite{Henniart-MathNachr1992} since it allows one to isolate the purely formal part of the argument from the part that relies on the structure theory of $p$-adic reductive groups.

\subsection{Setup} Let $G$ be a Hausdorff locally compact totally disconnected topological group. We write $Z$ for the center of $G$ and $C_G(g)$ for the centralizer of a given element $g\in G$. Consider an open subgroup $J\subset G$ such that $Z\subset J$ and $J/Z$ is compact. Let $\sg$ be a smooth finite dimensional representation of $J$ such that the compactly induced representation $\rho:=c-\Ind_J^G(\sg)$ is admissible. For every $g\in G$, we will write $n_J(g)$ for the number of right cosets $Jx$ of $J$ in $G$ such that $xgx^{-1}\in J$; thus $n_J(g)$ is either a nonnegative integer or $+\infty$.

\begin{thm}\label{t:Frobenius-char-formula}
Let $g\in G$ and fix a left Haar measure $\mu$ on $G$.

\begin{enumerate}[$($a$)$]
\item Suppose that $C_G(g)/Z$ is compact, $n_J(g)<\infty$ and there is an open neighborhood $U$ of $g$ in $G$ such that $n_J(y)=n_J(g)$ for all $y\in U$. Then for every open subgroup $N'\subset G$, there exists a compact open subgroup $N\subset N'$ such that
    \begin{equation}\label{e:Frobenius-character-formula}
    \frac{1}{\mu(N)}\cdot \tr\left(\int_{z\in N} \rho(gz)\,d\mu(z)\right) = \sum_{x\in J\setminus G,\ xgx^{-1}\in J} \tr(\sg(xgx^{-1})).
    \end{equation}
 \sbr
\item Let $F$ be a nondiscrete locally compact non-Archimedean field and $G=GL_n(F)$. Then the assumptions of (a) are satisfied for every regular elliptic $g\in G$.
\end{enumerate}
\end{thm}

\begin{rems}
\begin{enumerate}[(1)]
\item The right hand side of formula \eqref{e:Frobenius-character-formula} makes sense if $n_J(g)<\infty$. The integral on the left hand side is the same as $\rho(f)$, where $f\in C^\infty_c(G)$ is the indicator function of the coset $gN$. Since $\rho$ is assumed to be admissible, the left hand side of \eqref{e:Frobenius-character-formula} also makes sense.
 \sbr
\item In the situation of part (b), part (a) is equivalent to the formula proved in the appendix of \cite{Henniart-MathNachr1992}. Indeed, the character of $\rho$ (viewed as a generalized function on $G$) is given by a locally constant function on the set of regular semisimple elements of $G$, and if $N$ is any sufficiently small compact open subgroup of $G$, then the value of that function at $g$ is equal to the left hand side of \eqref{e:Frobenius-character-formula}.
\end{enumerate}
\end{rems}

\subsection{Proof of part (a)} Write $n=n_J(g)$ and let $Jx_1,Jx_2,\dotsc,Jx_n$ be all the right cosets of $J$ in $G$ such that $x_i gx_i^{-1}\in J$. Since $J$ is open, by shrinking $U$ if necessary we may assume that $x_i Ux_i^{-1}\subset J$ for every $i$. It follows that if $y\in U$, then $x_i yx_i^{-1}\in J$ for every $i$, so since $n_J(y)=n$ by assumption, we have $xyx^{-1}\not\in J$ whenever $Jx$ is not equal to one of the cosets $Jx_i$. This means that after shrinking $U$, we may assume that for every $x\in G$, we have either $xUx^{-1}\subset J$ or $xUx^{-1}\cap J=\varnothing$.

\mbr

Let $N'\subset G$ be an open subgroup. Shrinking $N'$ if necessary, we may assume that $N'$ is compact, $gN'\subset U$ and $x_i N'x_i^{-1}\subset\Ker(\sg)\subset J$ for each $1\leq i\leq n$, where the $x_i$ are as in the previous paragraph. Since $C_G(g)/Z$ is compact and $N'$ is open in $G$, the normalizer of $N'$ in $C_G(g)$ has finite index in $C_G(g)$, so there exists a compact open subgroup $N\subset N'$ normalized by $C_G(g)$. Let us prove that formula \eqref{e:Frobenius-character-formula} holds for this subgroup $N$. Form $H=C_G(g)\cdot N$; this is an open subgroup of $G$ and $H/Z$ is compact. By a standard argument, the restriction of $\rho$ to $H$ decomposes as
\[
\rho\bigl\lvert_H = \bigoplus_{x\in J\setminus G/H} \Ind_{H\cap x^{-1}Jx}^H \bigl(\sg^x\bigl\lvert_{H\cap x^{-1}Jx}\bigr)
\]
where for every $x\in G$, we write $\sg^x$ for the representation of $x^{-1}Jx$ defined by the formula $\sg^x(\ga)=\sg(x\ga x^{-1})$; there is no need to use compact induction here because $H\cap x^{-1}Jx$ has finite index in $H$. Write $\rho_x:=\Ind_{H\cap x^{-1}Jx}^H \bigl(\sg^x\bigl\lvert_{H\cap x^{-1}Jx}\bigr)$. The character of $\rho_x$ can be calculated via the usual Frobenius formula:
\[
\tr(\rho_x(h)) = \sum_{\ga\in (H\cap x^{-1}Jx)\setminus H} \tr\bigl(\sg(x\ga h\ga^{-1}x^{-1})\bigr).
\]

\mbr

Let us calculate $\tr(\rho_x(gz))$ for each $z\in N$. We note that since $N$ is normal in $H$ and $H=C_G(g)N$, the coset $gN$ is stable under $H$-conjugation. There are two possibilities: either $gN\cap x^{-1}Jx=\varnothing$, in which case we find that $\tr(\rho_x(gz))=0$ for all $z\in N$, or $Jx=Jx_i$ for some $i$, in which case, by the previous part of the proof, we find that $gN\subset x^{-1}Jx$ and $N\subset\Ker(\sg^x)\subset x^{-1}Jx$; this $g\in x^{-1}Jx$ as well. In the latter case, we obtain $\tr(\rho_x(gz))=[H:H\cap x^{-1}Jx]\cdot\tr(\sg^x(g))$ for all $z\in N$.

\mbr

Let $\e_{gN}$ be the indicator function of $gN$ and $f:=\frac{1}{\mu(N)}\cdot\e_{gN}\in C^\infty_c(H)\subset C^\infty_c(G)$. The left hand side of \eqref{e:Frobenius-character-formula} is equal to $\tr(\rho(f))$, which is the same as $\sum\limits_{x\in J\setminus G/H}\tr(\rho_x(f))$. By the previous paragraph, given $x\in G$, we have $\tr(\rho_x(f))=0$ if $g\not\in x^{-1}Jx$ and $\tr(\rho_x(f))=[H:H\cap x^{-1}Jx]\cdot\tr(\sg(xgx^{-1}))$ if $g\in x^{-1}Jx$. Therefore
\begin{eqnarray*}
\sum\limits_{x\in J\setminus G/H}\tr(\rho_x(f)) &=& \sum_{x\in J\setminus G/H,\ xgx^{-1}\in J}[H:H\cap x^{-1}Jx]\cdot\tr(\sg(xgx^{-1})) \\ &=& \sum_{x\in J\setminus G,\ xgx^{-1}\in J} \tr(\sg(xgx^{-1})),
\end{eqnarray*}
which proves (a).

\subsection{Proof of part (b)} Let $g\in G=GL_n(F)$ be a regular elliptic element. Define $K=F[g]$ as the $F$-subalgebra of $Mat_n(F)$ generated by $g$. Then $K$ is a separable field extension of $F$ of degree $n$ and $C_G(g)\cong K^\times$, so $C_G(g)/Z\cong K^\times/F^\times$ is compact.

\begin{lem}\label{l:auxiliary-HC}
$n_J(g)<\infty$ and there exists an open neighborhood $\Om$ of $g$ in $K^\times$ such that for each $x\in G$, we have either $x\Om x^{-1}\subset J$ or $x\Om x^{-1}\cap J=\varnothing$.
\end{lem}

If $\Om$ is as in the lemma, $n_J(y)=n_J(g)$ for all $y\in\Om$. The map $G\times K^\times\rar{}G$ given by $(\ga,y)\mapsto\ga y\ga^{-1}$ is open in a neighborhood of the point $(1,g)$, as one can easily check by calculating its differential (this idea is a simplified version of Harish-Chandra's submersion principle \cite{Harish-Chandra-submersion}). In particular, there exists an open neighborhood $U$ of $g$ in $G$ such that every point of $U$ is $G$-conjugate to a point of $\Om$. As the function $n_J$ is invariant under $G$-conjugation, we see that this function is constant on $U$, which completes the proof of Theorem \ref{t:Frobenius-char-formula}.

\subsection{Proof of Lemma \ref{l:auxiliary-HC}} We repeat the first paragraph of the proof of \cite[Thm.~A2]{Henniart-MathNachr1992}. As $J/F^\times$ is compact, there exists a compact subset $C'\subset G$ such that $J\subset C'\cdot F^\times$. By Lemma 19 (page 52) of \cite{harish-chandra-van-dijk-notes}, there exist an open neighborhood $\Om$ of $g$ in $K^\times$ and a compact subset $C^*\subset G/K^\times$ such that if $x\in G$ satisfies $x\Om x^{-1}\cap(C'\cdot F^\times)\neq\varnothing$, then the image of $x$ in $G/K^\times$ lies in $C^*$. Since $K^\times/F^\times$ is compact, there exists a compact subset $C\subset G$ such that the inverse image of $C^*$ in $G$ is contained in $C\cdot F^\times$. \emph{A fortiori}, if $x\in G$ satisfies $x\Om x^{-1}\cap J\neq\varnothing$, then $x\in C\cdot F^\times$. Since $G/J$ is discrete, the image of $C$ in $G/J$ is finite. Since $F^\times\subset J$, we deduce that there exist only finitely many right cosets $Jx$ of $J$ in $G$ such that $x\Om x^{-1}\cap J\neq\varnothing$. In particular, $n_J(g)<\infty$, and after shrinking $\Om$, we can ensure that $x\Om x^{-1}\cap J\neq\varnothing$ if and only if $xgx^{-1}\in J$. Finally, shrinking $\Om$ again, we can also ensure that if $xgx^{-1}\in J$, then $x\Om x^{-1}\subset J$, which completes the proof.

\section{Proof of the geometric theorem}\label{s:proof-t:main-geometric-theorem}

In this section we prove Theorem \ref{t:main-geometric-theorem}. Recall that $q$ denotes a power of a prime number $p$ and $n,r_0\geq 2$ are integers. In \S\ref{ss:geometric-ingredient} we defined an algebraic group $\cR^\times$ over $\fqn$ with unipotent radical $U\subset\cR^\times$ and a smooth hypersurface $X=L_{q^n}^{-1}(Y)\subset U$.

\subsection{Preliminary reductions} From now on, to simplify the notation, we adopt the following convention. If $S$ is any variety over $\fqn$ and $\cL$ is a local system on $S$, we will simply write $H^j_c(S,\cL)$ for the $j$-th compactly supported cohomology of $S\tens_{\fqn}\bfq$ with coefficients in the local system obtained from $\cL$ by pullback. Since the $\Gal(\bfq/\fqn)$-action is not considered here, we view $H^j_c(S,\cL)$ merely as a finite dimensional $\ql$-vector space. (It is \emph{not} the same as the compactly supported cohomology of $S$ with coefficients in $\cL$ in the usual sense.) This convention applies in particular to the constant rank $1$ local system $\cL=\ql$.

\mbr

Define $d=\ceil{\frac{n-1}{2}}$, so $d=(n-1)/2$ if $n$ is odd and $d=n/2$ if $n$ is even. Let $H\subset U$ be the subgroup given by $H(B)=\{1+a_{d+1}e_{d+1}+\dotsc+a_ne_n\st a_j\in B\}$ for any commutative $\fqn$-algebra $B$ (with the notation of \S\ref{ss:geometric-ingredient}). It is normal in $U$.

\mbr

The map $\pr_n:H\rar{}\bG_a$ given by $1+\sum_{j=d+1}^n a_j e_j\mapsto a_n$ is an algebraic group homomorphism. In particular, if $\psi:\fqn\rar{}\qls$ is an additive character, then $\psi\circ \pr_n$ is a character of $H(\fqn)$. Recall also that we can identify $\fqn$ with the center of $U(\fqn)$. The proofs of the next facts are rather straightforward, so we skip them.

\begin{lem}\label{l:exercise}
Let $\psi:\fqn\to\qls$ be a character with trivial $\Gal(\fqn/\bF_q)$-stabilizer.

\begin{enumerate}[$($a$)$]
\item Every character of $H(\fqn)$ whose restriction to $\fqn\cong\{1+a_ne_n\}\subset H(\fqn)$ agrees with $\psi$ is $U(\fqn)$-conjugate to $\psi\circ \pr_n$.
 \sbr
\item If $n$ is odd, the normalizer of $\psi\circ \pr_n$ in $U(\fqn)$ equals $H(\fqn)$.
 \sbr
\item If $n$ is even, the normalizer of $\psi\circ \pr_n$ in $U(\fqn)$ equals the subgroup $H^+(\fqn)$, where $H^+\subset U$ is given by $H^+(B)=\{1+a_de_d+\dotsc+a_ne_n\st a_j\in B\}$ for any commutative $\fqn$-algebra $B$.
\end{enumerate}
\end{lem}

\begin{lem}\label{l:exercise2}
Let $\psi$ be as in Lemma \ref{l:exercise}.

\begin{enumerate}[$($a$)$]
\item If $n$ is odd, $\pi:=\Ind_{H(\fqn)}^{U(\fqn)}(\psi\circ \pr_n)$ is an irreducible representation of $U(\fqn)$.
 \sbr
\item If $n$ is even, $H^+(\fqn)$ has a unique irreducible representation $\sg$ whose restriction to $H(\fqn)$ is a direct sum of copies of $\psi\circ \pr_n$. Moreover, $\sg$ is $q^{n/2}$-dimensional, $\pi:=\Ind_{H^+(\fqn)}^{U(\fqn)}(\sg)$ is irreducible and $\Ind_{H(\fqn)}^{U(\fqn)}(\psi\circ \pr_n)$ is a direct sum of $q^{n/2}$ copies of $\pi$.
 \sbr
\item In both cases, $\pi$ is the unique irreducible representation of $U(\fqn)$ with central character $\psi$.
\end{enumerate}
\end{lem}

The last lemma implies that parts (a) and (b) of Theorem \ref{t:main-geometric-theorem} follow from

\begin{prop}
Let $\psi$ be as in Lemma \ref{l:exercise}.

\begin{enumerate}[$($a$)$]
\item We have
\[
\dim\Hom_{U(\fqn)}\bigl( \Ind_{H(\fqn)}^{U(\fqn)}(\psi\circ\pr_n), H^{n-1}_c(X,\ql) \bigr) = \begin{cases} 1 & \text{if } n \text{ is odd}, \\ q^{n/2} & \text{if } n \text{ is even}.
\end{cases}
\]
 \sbr
\item If $j\neq n-1$, then $\Hom_{U(\fqn)}\bigl( \Ind_{H(\fqn)}^{U(\fqn)}(\psi\circ\pr_n), H^{j}_c(X,\ql) \bigr)=0$.
\end{enumerate}
\end{prop}

To prove the last proposition we use the methods developed in \cite[\S2]{DLtheory}. We identify the homogeneous space $U/H$ with the affine space of dimension $d$ in the evident way, and we write $s:U/H\rar{}U$ for the natural section given by the formula $s(a_1,\dotsc,a_d)=1+\sum_{j=1}^d a_j e_j$. Consider the morphism
\begin{equation}\label{e:the-map-f}
f : (U/H) \times H \rar{} U, \qquad (x,h)\mapsto \Fr_{q^n}(s(x))\cdot h\cdot s(x)^{-1}.
\end{equation}
Let us also write $\widetilde{p}_n=\pr_n\circ\pr_2:(U/H)\times H\rar{}\bG_a$ for the projection onto the last coordinate (where $\pr_2$ is the second projection and $\pr_n$ was defined earlier). Then by \cite[Prop.~2.3]{DLtheory}, for any character $\psi$ of $\fqn$, we have a vector space isomorphism
\[
\Hom_{U(\fqn)}\bigl( \Ind_{H(\fqn)}^{U(\fqn)}(\psi\circ\pr_n), H^{j}_c(X,\ql) \bigr) \cong H^j_c(f^{-1}(Y),\widetilde{p}_n^*(\cL_\psi))
\]
where $Y\subset U$ is the hyperplane defined by $a_n=0$ and $\cL_\psi$ is the Artin-Schreier local system on $\bG_a$ corresponding to the character $\psi$.

\mbr

We finally see that the proof of parts (a) and (b) of Theorem \ref{t:main-geometric-theorem} is reduced to

\begin{prop}\label{p:cohomology-calculation}
If $\psi$ is as in Lemma \ref{l:exercise}, then
\[
\dim H^j_c(f^{-1}(Y),\widetilde{p}_n^*(\cL_\psi)) = \begin{cases} 0 & \text{if } j\neq n-1, \\ 1 & \text{if } j=n-1 \text{ and } n \text{ is odd}, \\ q^{n/2} & \text{if } j=n-1 \text{ and } n \text{ is even}.
\end{cases}
\]
\end{prop}

The proof of the last proposition is contained in \S\S\ref{ss:inductive-setup}--\ref{ss:induction-base}.

\subsection{Inductive setup}\label{ss:inductive-setup} We have $f^{-1}(Y)=(\pr_n\circ f)^{-1}(0)$, where $\pr_n:U\rar{}\bG_a$ is the projection onto the last coordinate. Let us identify $(U/H)\times H$ with $\bA^n$ in the natural way. Under this identification, the map \eqref{e:the-map-f} becomes
\begin{eqnarray*}
f(a_1,\dotsc,a_n) &=& (1+a_1^{q^n}e_1+\dotsc+a_d^{q^n}e_d)\cdot(1+a_{d+1}e_{d+1}+\dotsc+a_ne_n) \\
&& \cdot (1+a_1e_1+\dotsc+a_de_d)^{-1}
\end{eqnarray*}
Therefore there exists a polynomial map $\al:\bA^{n-1}\rar{}\bG_a$ such that
\[
\pr_n(f(a_1,\dotsc,a_n)) = a_n-\al(a_1,\dotsc,a_{n-1}).
\]
This observation implies the following

\begin{lem}\label{l:reduction-to-affine-space}
\begin{enumerate}[$($a$)$]
\item The projection map $(U/H)\times H\rar{}\bA^{n-1}$ onto the first $(n-1)$ coordinates identifies $f^{-1}(Y)$ with $\bA^{n-1}$.
 \sbr
\item Under this identification, the local system $\widetilde{p}_n^*(\cL_\psi)\bigl\lvert_{f^{-1}(Y)}$ corresponds to $\al^*(\cL_\psi)$.
\end{enumerate}
\end{lem}

Therefore we must compute $H^i_c(\bA^{n-1},\al^*(\cL_\psi))$ for all $i$. To this end, for every $1\leq j\leq d$, we consider the morphism $f_j:\bA^{n-2(j-1)}\rar{}U$ given by
\[
(a_j,a_{j+1},\dotsc,a_{n-j},a_n)\mapsto f(\underbrace{0,\dotsc,0}_{j-1},a_j,a_{j+1},\dotsc,a_{n-j},\underbrace{0,\dotsc,0}_{j-1},a_n)
\]
We have
\begin{eqnarray*}
f_j(a_j,a_{j+1},\dotsc,a_{n-j},a_n) &=& (1+a_j^{q^n}e_j+\dotsc+a_d^{q^n}e_d) \\ && \cdot(1+a_{d+1}e_{d+1}+\dotsc+a_{n-j}e_{n-j}+a_ne_n) \\
&& \cdot (1+a_je_j+\dotsc+a_de_d)^{-1}
\end{eqnarray*}
Therefore there exists a polynomial map $\al_j:\bA^{n-2j+1}\rar{}\bG_a$ such that
\[
\pr_n(f_j(a_j,a_{j+1},\dotsc,a_{n-j},a_n)) = a_n-\al_j(a_j,a_{j+1},\dotsc,a_{n-j}).
\]
By construction, $\al_1=\al$ and $f_1=f$.

\begin{lem}\label{l:induction-step}
$H^i_c(\bA^{n-2j+1},\al_j^*\cL_\psi)\cong H^{i-2}_c(\bA^{n-2j-1},\al_{j+1}^*\cL)$ for $1\leq j<d$ and all $i$.
\end{lem}

\begin{proof}
$\pr_n(f_j(a_j,a_{j+1},\dotsc,a_{n-j},a_n))$ is a linear combination of monomials in the variables $a_j,a_{j+1},\dotsc,a_{n-j},a_n$. By inspection, only two of these monomials involve the variable $a_{n-j}$, namely, $a_j^{q^n}a_{n-j}^{q^j}$ and $-a_{n-j}a_j^{q^{n-j}}$. Hence we can write
\begin{eqnarray*}
\al_j(a_j,a_{j+1},\dotsc,a_{n-j}) &=& a_{n-j} a_j^{q^{n-j}} - a_j^{q^n} a_{n-j}^{q^j} \\ && + \de_{j+1}(a_{j+1},\dotsc,a_{n-j-1}) \\ && + a_j\cdot \be_j(a_j,\dotsc,a_{n-j-1})
\end{eqnarray*}
for suitable polynomial maps $\de_{j+1}:\bA^{n-2j-1}\rar{}\bG_a$ and $\be_j:\bA^{n-2j}\rar{}\bG_a$. To determine $\de_{j+1}$ we substitute $a_j=a_{n-j}=0$ into the last identity and find that
\[
\de_{j+1}(a_{j+1},\dotsc,a_{n-j-1})=\al_j(0,a_{j+1},\dotsc,a_{n-j-1},0)=\al_{j+1}(a_{j+1},\dotsc,a_{n-j-1}).
\]
It remains to apply \cite[Prop.~2.10]{DLtheory} with $S_2=\bA^{n-2j}$ (we use $a_j,a_{j+1},\dotsc,a_{n-j-1}$ as the coordinates on $S_2$ and put $y=a_{n-j}$), $P=\al_j$ and
\[
P_2(a_j,a_{j+1},\dotsc,a_{n-j-1})=\al_{j+1}(a_{j+1},\dotsc,a_{n-j-1}) + a_j\cdot \be_j(a_j,\dotsc,a_{n-j-1}).
\]
With the notation of \emph{loc.~cit.}, we can identify $S_3$ with $\bA^{n-2j-1}$ using the coordinates $a_{j+1},\dotsc,a_{n-j-1}$, so that $P_3$ becomes identified with $\al_{j+1}$, completing the proof.
\end{proof}

\subsection{Base of induction}\label{ss:induction-base} Applying Lemma \ref{l:induction-step} $(d-1)$ times (with $j=1,2,\dotsc,d-1$) and using the isomorphisms $H^i_c(f^{-1}(Y),\widetilde{p}_n^*(\cL_\psi))\cong H^i_c(\bA^{n-1},\al^*(\cL_\psi))$ resulting from Lemma \ref{l:reduction-to-affine-space}, we see that Proposition \ref{p:cohomology-calculation} follows from

\begin{lem}\label{l:induction-base}
\begin{enumerate}[$($a$)$]
\item Suppose that $n$ is odd, so that $n=2d+1$. Then
\begin{equation}\label{e:induction-base-odd}
\dim H^i_c(\bA^2,\al_d^*\cL_\psi) = \begin{cases}
1 & \text{if } i=2, \\
0 & \text{otherwise}.
\end{cases}
\end{equation}
 \sbr
\item Suppose that $n$ is even, so that $n=2d$. Then
\begin{equation}\label{e:induction-base-even}
\dim H^i_c(\bA^1,\al_d^*\cL_\psi) = \begin{cases}
q^{n/2} & \text{if } i=1, \\
0 & \text{otherwise}.
\end{cases}
\end{equation}
\end{enumerate}
\end{lem}

\begin{proof}
(a) Let us calculate the morphism $\al_d:\bA^2\rar{}\bG_a$. We have
\[
f_d(a_d,a_{d+1},a_n) = (1+a_d^{q^n}e_d)\cdot(1+a_{d+1}e_{d+1}+a_ne_n)\cdot (1+a_de_d)^{-1}.
\]
Using the definition of the product in $\cR^\times$, we find
\[
(1+a_de_d)^{-1} = \begin{cases}
1-a_1e_1+a_1^{1+q}e_2-a_1^{1+q+q^2}e_3 & \text{if } r_0=2,\ n=3; \\
1-a_de_d+a_d^{1+q^d}e_{2d} & \text{if } r_0=2,\ n>3; \\
1-a_de_d & \text{if } r_0>2.
\end{cases}
\]
Hence
\[
\al_d(a_d,a_{d+1}) = \begin{cases}
a_1^{q^2} a_2 - a_1^{q^3}a_2^q + a_1^{1+q+q^2} - a_1^{q+q^2+q^3} & \text{if } r_0=2,\ n=3; \\
a_d^{q^{d+1}}a_{d+1}-a_d^{q^n}a_{d+1}^{q^d} & \text{otherwise}.
\end{cases}
\]
We apply \cite[Prop.~2.10]{DLtheory} with $S_2=\bA^1$ (using the coordinate $x=a_d$ and writing $y=a_{d+1}$) and $P(x,y)=\al_d(x,y)$. We have $P_2(x)=x^{1+q+q^2}-x^{q+q^2+q^3}$ in the first case and $P_2(x)=0$ in the second case. In both cases $S_3=\{0\}\subset\bA^1=S_2$. So $P_3^*\cL_\psi$ is the trivial rank $1$ local system on $\Spec(\fqn)$, which yields \eqref{e:induction-base-odd}.

\mbr

(b) If $n=2d$ is even, the morphism $f_d:\bA^2\rar{}U$ is given by
\begin{eqnarray*}
f_d(a_d,a_n) &=& (1+a_d^{q^n}e_d)\cdot(1+a_ne_n)\cdot(1+a_de_d)^{-1} \\
&=& (1+a_d^{q^n}e_d+a_ne_n)\cdot(1-a_de_d+a_d^{1+q^d}e_n) \\
&=& 1 + (a_d^{q^n}-a_d)e_d + \bigl(a_n+a_d^{1+q^d}-a_d^{q^d+q^n}\bigr)e_n,
\end{eqnarray*}
which implies that $\al_d:\bA^1\rar{}\bG_a$ is given by $\al_d(x)=x^{q^n+q^d}-x^{q^d+1}$. Put $q'=q^d$, so that $q^n=q'^2$, define $\ga':\bA^1\rar{}\bG_a$ by $\ga'(x)=x^{1+q'}$ and let $\psi':\fqn\rar{}\qls$ be the character given by $\psi'(z)=\psi(z^{q'}-z)$. Note that $\psi'$ is a nontrivial character. Moreover, if $\cL_{\psi'}$ is the Artin-Schreier local system on $\bG_a$ over $\fqn$ corresponding to $\psi'$, then $\al_d^*(\cL_\psi)\cong\ga'^*(\cL_{\psi'})$ by construction. To compute $H^i_c(\bA^1,\ga'^*(\cL_{\psi'}))$ we embed $\bA^1\cong\bG_a$ into $\bP^1$ and view $\ga'$ as a morphism $\bP^1\rar{}\bP^1$. The Swan conductor of $\cL_{\psi'}$ at $\infty\in\bP^1$ equals $1$, so the Swan conductor of $\ga'^*\cL_{\psi'}$ at $\infty$ equals $q'+1$ (since $\ga'$ is tamely ramified and its ramification index at $\infty$ is $q'+1$). In particular, $\ga'^*(\cL_{\psi'})$ is nontrivial, so $H^2_c(\bA^1,\ga'^*(\cL_{\psi'}))=0$. By the Grothendieck-Ogg-Shafarevich formula\footnote{A more elementary approach is possible, and it is used in \cite{maximal-varieties-LLC}. However, it is a bit longer.} \cite[(3.2.1)]{sga4.5}, $\chi_c(\bA^1,\ga'^*(\cL_{\psi'}))=1-(q'+1)=-q'$, which yields \eqref{e:induction-base-even}.
\end{proof}

\subsection{Proof of Theorem \ref{t:main-geometric-theorem}(c)} For every $a\in\fqn$, the element $\ze+\ze a e_n\in\cR^\times(\fqn)$ acts on $X$ and hence on the cohomology $H^j_c(X,\ql)$. The elements $\ze$ and $1+ae_n$ commute, $\ze$ has order prime to $p$ and $1+ae_n$ has order a power of $p$ (in fact, either $1$ or $p$). By the fixed point formula of \cite[\S3]{deligne-lusztig}, we have
\[
\sum_j (-1)^j\cdot \tr\bigl( (\ze+\ze ae_n)^*; H^j_c(X,\ql) \bigr) = \sum_j (-1)^j\cdot \tr\bigl( (1+ae_n)^*; H^j_c(X^\ze,\ql) \bigr)
\]
where by abuse of notation we write $X$ and $X^\ze$ in place of $X\tens_{\fqn}\bfq$ and $(X\tens_{\fqn}\bfq)^\ze$, respectively. The (right) action of $\ze$ on $X$ is given by the formula
\[
(1+a_1e_1+\dotsc+a_ne_n)\bullet\ze = \ze^{-1}(1+a_1e_1+\dotsc+a_ne_n)\ze=1+\sum_{j=1}^n a_j\ze^{q^j-1}e_j,
\]
and since $\ze$ has trivial stabilizer in $\Gal(\fqn/\bF_q)$, we see that $(X\tens_{\fqn}\bfq)^\ze$ is equal to the finite discrete set of points of the form $1+b\tau$, where $b\in\fqn$. The element $1+ae_n\in\cR^\times(\fqn)$ acts on this set by translation: $b\mapsto b+a$. Hence
\[
\sum_j (-1)^j\cdot \tr\bigl( (\ze+\ze ae_n)^*; H^j_c(X,\ql) \bigr) =
\begin{cases}
q^n & \text{if } a=0, \\
0 & \text{otherwise}.
\end{cases}
\]
Multiplying the last identity by $\psi(a)$ and averaging over all $a\in\fqn$, we obtain Theorem \ref{t:main-geometric-theorem}(c) in view of Theorem \ref{t:main-geometric-theorem}(b).

\section{Proofs of the algebraic theorems}

\subsection{Some preliminaries}\label{ss:preliminaries-main} We first explain why part (c) of each of Theorems \ref{t:GL-n} and \ref{t:division-algebra} follows from parts (a) and (b).

\mbr

The fact that $\pi$ (resp. $\rho$) has central character $\te\bigl\lvert_{F^\times}$ follows from the observation that $\widetilde{\te}\bigl\lvert_{F^\times}=\te\bigl\lvert_{F^\times}$ in the situation of part (a), and $\sg\bigl\lvert_{F^\times}$ is a direct sum of copies of $\te\bigl\lvert_{F^\times}$ in the situation of part (b).

\mbr

The fact that the isomorphism class of $\pi$ (resp. $\rho$) is invariant under twist by the character $\eps\circ\det_A$ (resp. $\eps\circ\Nrd_{D/F}$) follows from the observation that since the extension $E\supset F$ is unramified, we have $\eps\bigl\lvert_{\cO_F^\times}\equiv 1$, and therefore $\eps\circ\det_A$ (resp. $\eps\circ\Nrd_{D/F}$) is trivial on $E^\times\cdot\cO_A^\times$ (resp. $E^\times\cdot\cO_D^\times$).

\mbr

To prove the character identities of parts (c) of Theorems \ref{t:GL-n} and \ref{t:division-algebra}, we recall that $\sG=\Gal(E/F)$ can be naturally viewed as a subgroup of $\cO_A^\times\subset A^\times$. The normalizer of $E^\times$ in $G=A^\times$ is equal to $\sG\cdot E^\times$. On the other hand, let $N_{D^\times}(E^\times)$ denote the normalizer of $E^\times$ in $D^\times$. The Skolem-Noether theorem implies that the conjugation action of $N_{D^\times}(E^\times)$ induces an isomorphism $N_{D^\times}(E^\times)/E^\times\rar{\simeq}\sG$. In view of these remarks, the required character identities follow from the formulas of Theorem \ref{t:supercuspidals-GL-n} (resp. Proposition \ref{p:division-algebra-auxiliary}) together with

\begin{lem}\label{l:1}
Let $m\geq 1$ be an integer and let $x\in\cO_E^\times$ be very regular.

\begin{enumerate}[$($a$)$]
\item If $g\in G$ and $gxg^{-1}\in E^\times\cdot U^m_G$, then $g\in\sG\cdot E^\times\cdot U^m_G$.
 \sbr
\item If $g\in D^\times$ and $gxg^{-1}\in E^\times\cdot U^m_D$, then $g\in N_{D^\times}(E^\times)\cdot U^m_D$.
\end{enumerate}
\end{lem}

\begin{proof}
(a) By Lemma \ref{l:conjugate-very-regular}, we have $g\in F^\times\cdot\cO_A^\times$. Then $gxg^{-1}\in\cO_A^\times\cap E^\times\cdot U^m_G$, which implies that $gxg^{-1}\in y\cdot U^m_G$ for some very regular $y\in\cO_E^\times$. By Lemma \ref{l:2}(a), after multiplying $g$ on the left by an element of $U_G^m$, we may assume that $gxg^{-1}\in\cO_E^\times$ and is very regular. But then $E=F[x]=F[gxg^{-1}]$, so $g\in N_G(E^\times)=\sG\cdot E^\times$.

\mbr

(b) We automatically have $gxg^{-1}\in\cO_D^\times\cap E^\times\cdot U^m_D$, and the rest of the proof proceeds in the same way, using Lemma \ref{l:2}(b) instead of Lemma \ref{l:2}(a).
\end{proof}

\begin{lem}\label{l:2}
Let $m\geq 1$ be an integer and let $x\in\cO_E^\times$ be very regular.

\begin{enumerate}[$($a$)$]
\item Every element of $x\cdot U^m_G$ is $U^m_G$-conjugate to an element of $x\cdot U^m_E$.
 \sbr
\item Every element of $x\cdot U^m_D$ is $U^m_D$-conjugate to an element of $x\cdot U^{\ceil{m/n}}_E$.
\end{enumerate}
\end{lem}

\begin{proof}
(a) It suffices to show that every element of $x\cdot U^m_G$ is $U^m_G$-conjugate to an element of $x\cdot U^m_E\cdot U^{m+1}_G$. To this end, we may work in the quotient group $\cO_A^\times/U^{m+1}_G$. Each element of $x\cdot U^m_G$ has the form $x+\varpi^m y$ for some $y\in\cO_A$. Given $y\in\cO_A$, we must find $z\in\cO_A$ and $y'\in\cO_E$ such that
\[
(1+\varpi^m z)\cdot(x+\varpi^m y) \equiv (x+\varpi^m y')\cdot(1+\varpi^m z) \mod P^{m+1}_A.
\]
The last identity is equivalent to $y\equiv y'+xz-zx\mod P_A$, so the existence of $y'$ and $z$ results from the following observation. Let $\bar{x}$ denote the image of $x$ in $\cO_E/P_E=\fqn$ and identify $\cO_A/P_A$ with $\overline{A}:=\End_{\bF_q}(\fqn)$. Then the $\bF_q$-linear operator $\ad\bar{x}$ on $\overline{A}$ is semisimple and its kernel is equal to $\fqn\subset\overline{A}$ because $x$ is very regular.

\mbr

(b) It suffices to show that every element of $x\cdot U^m_D$ is $U^m_D$-conjugate to an element of $x\cdot U^{\ceil{m/n}}_E\cdot U^{m+1}_D$. If $n$ divides $m$, this follows from the fact that $U^m_D\subset U^{\ceil{m/n}}_E\cdot U^{m+1}_D$ in this case, so let us assume that $n$ does not divide $m$. We may work in the quotient group $\cO_D^\times/U^{m+1}_D$. Each element of $x\cdot U^m_D$ has the form $x+y\Pi^m$ for some $y\in\cO_E$. Given $y\in\cO_E$, we must find $z\in\cO_E$ such that
\[
(1+z\Pi^m)\cdot(x+y\Pi^m)\equiv x\cdot(1+z\Pi^m) \mod P_D^{m+1}.
\]
Since $x\in\cO_E^\times$ is very regular and $n$ does not divide $m$, we have $x-\vp^m(x)=x-\Pi^m x\Pi^{-m}\in\cO_E^\times$, so we can take $z=y\cdot(x-\vp^m(x))^{-1}$.
\end{proof}

\begin{cor}\label{c:conjugation-restrictions}
Let $m\geq 1$ be an integer and let $x\in\cO_E^\times$ be very regular.

\begin{enumerate}[$($a$)$]
\item If $y,z\in x\cdot U^m_G$ and $gyg^{-1}=z$, then $g\in E^\times\cdot U^m_G$.
 \sbr
\item If $y,z\in x\cdot U^m_D$ and $gyg^{-1}=z$, then $g\in E^\times\cdot U^m_D$.
\end{enumerate}
\end{cor}

\begin{proof}
(a) Lemma \ref{l:2}(a) shows that after multiplying $g$ on the left and on the right by elements of $U^m_G$, we may assume that $y,z\in x\cdot U^m_E$. In particular, $y,z$ are both very regular elements of $\cO_E^\times$ and $y\equiv x\equiv z\mod P_E$. Thus $E=F[y]=F[z]$, so $g\in N_G(E^\times)=\sG\cdot E^\times$. The fact that $y\equiv x\equiv z\mod P_E$ forces $g\in E^\times$ since $x$ is very regular. The proof of (b) is essentially identical to the proof of (a).
\end{proof}

\subsection{Trace pairings}\label{ss:trace-pairings} Let $\tr_A:A\rar{}F$ and let $\Trd_{D/F}:D\rar{}F$ denote the usual trace and the reduced trace, respectively. The symmetric bilinear forms $(a,b)\mapsto\tr_A(ab)$ and $(a,b)\mapsto\Trd_{D/F}(ab)$ on $A$ and $D$, respectively, are nondegenerate. Hence if $\psi_0:F\rar{}\qls$ is a nontrivial additive character, the maps $(a,b)\mapsto\psi_0\circ\tr_A(ab)$ and $(a,b)\mapsto\psi_0\circ\Trd_{D/F}(ab)$ allow us to identify the additive groups of $A$ and $D$ with their own Pontryagin duals. If $V$ is an additive subgroup of $A$ or $D$, we will denote by $V^\perp$ its annihilator in the Pontryagin dual of $A$ (resp. $D$), identified with an additive subgroup of $A$ (resp. $D$) in the way we just described.

\begin{lem}\label{l:properties-pairings}
\begin{enumerate}[$($a$)$]
\item We have $(gVg^{-1})^\perp=g(V^\perp)g^{-1}$ for all $g\in G$ $($resp. $g\in D^\times${}$)$.
 \sbr
\item Assume that $\psi_0$ has level $r_0$, i.e., $\psi_0$ is trivial on $P_F^{r_0}$ and nontrivial on $P_F^{r_0-1}$. Then $(P_A^m)^\perp=P_A^{r_0-m}$ and $(P_D^m)^\perp=P_D^{n(r_0-1)+1-m}$ for all $m\in\bZ$.
\end{enumerate}
\end{lem}

The proofs are straightforward, so we skip them.

\subsection{Intertwiners}\label{ss:intertwiners} We now turn to the results that will be used in the proofs of the irreducibility of the induced representations constructed in Theorems \ref{t:GL-n} and \ref{t:division-algebra}. From now on we work with a fixed primitive character $\te:E^\times\rar{}\qls$ of level $r_0\geq 2$. The arguments we use are very similar to those appearing in \cite{gerardin}.

\begin{lem}\label{l:3}
There exist a very regular element $y\in\cO_E^\times$ and an additive character $\psi_0:F\rar{}\qls$ of level $r_0$ such that $\te(1+a)=\psi_0\circ\Tr_{E/F}(ya)$ for all $a\in P_E^{r_0-1}$.
\end{lem}

This follows at once from the fact that $\te$ is primitive of level $r_0$. From now on we fix $y$ and $\psi_0$ satisfying the requirements of Lemma \ref{l:3} and use $\psi_0$ to identify the additive groups of $A$ and $D$ with their own Pontryagin duals, as explained in \S\ref{ss:trace-pairings}.

\mbr

Define $r=\ceil{r_0/2}$, so that $r=\frac{r_0}{2}$ when $r_0$ is even and $r=\frac{r_0+1}{2}$ when $r_0$ is odd. We fix an additive character $\al:E\rar{}\qls$ such that $\al(x)=\te(1+x)$ for all $x\in P^r_E$. Recall the decomposition $A=E\oplus C$, where $C$ is the orthogonal complement of $E$ with respect to the trace pairing on $A$. Let $\al_A$ denote the composition of $\al$ with the corresponding projection $A\rar{}E$. Similarly, we have the decomposition $D=E\oplus C'$, where $C'$ is the orthogonal complement of $E$ with respect to the reduced trace pairing on $D$, and we write $\al_D$ for the composition of $\al$ with the corresponding projection $D\rar{}E$. Then the formula $\te_A(1+x)=\al_A(x)$ defines a character $U^r_G\rar{}\qls$ that agrees with $\te$ on $U^r_G\cap E^\times=U^r_E$ and is trivial on $1+(C\cap P^r_A)$. Next define $s=\frac{n(r_0-1)}{2}+1$ when $r_0$ is odd and $s=\frac{nr_0}{2}$ when $r_0$ is even. Then the formula $\te_D(1+x)=\al_D(x)$ defines a character $U^s_D\rar{}\qls$ that agrees with $\te$ on $U^s_D\cap E^\times=U^r_E$ and is trivial on $1+(C'\cap P^s_D)$.

\begin{lem}\label{l:intertwiners} With the notation above and the terminology of Definition \ref{d:intertwiners},

\begin{enumerate}[$($a$)$]
\item if $g\in G$ intertwines the pair $(U^r_G,\te_A)$, then $g\in E^\times\cdot U^{r_0-r}_G$;
 \sbr
\item if $g\in D^\times$ intertwines the pair $(U^s_D,\te_D)$, then $g\in E^\times\cdot U^{n(r_0-1)+1-s}_D$.
\end{enumerate}
\end{lem}

\begin{proof}
(a) By construction, $\al_A(x)=\psi_0\circ\tr_A(xy)$ for all $x\in P^{r_0-1}_A+C$, which implies (using Lemma \ref{l:properties-pairings}(b)) that under the identification of $A$ with its own Pontryagin dual chosen above, $\al_A$ corresponds to an element $z\in y+(P_A\cap E)=y+P_E$. Now assume that $g\in G$ intertwines the pair $(U^r_G,\te_A)$. Then $\al_A(gxg^{-1})=\al_A(x)$ for all $x\in g^{-1}P_A^r g\cap P_A^r$, which means that $g^{-1}zg-z\in(g^{-1}P^r_Ag\cap P^r_A)^\perp=g^{-1}P^{r_0-r}_Ag+P^{r_0-r}_A$. Therefore we can find $z_1,z_2\in P^{r_0-r}_A$ with $g^{-1}(z+z_1)g=z+z_2$. But $z$ is a very regular element of $\cO_E^\times$ because $y$ is, so $g\in E^\times\cdot U^{r_0-r}_G$ by Corollary \ref{c:conjugation-restrictions}(a).

\mbr

The proof of (b) is very similar to that of (a), so we omit it.
\end{proof}

\subsection{Proof of Theorem \ref{t:GL-n}}\label{ss:proof-t:GL-n} We saw in \S\ref{ss:preliminaries-main} that it remains to establish parts (a) and (b) of the theorem. The first assertion of part (a) is obvious. To prove the second one, we observe that with the notation of \S\ref{ss:intertwiners}, we have $\widetilde{\te}\bigl\lvert_{U^{r_0/2}_G}=\te_A$. So Lemma \ref{l:intertwiners}(a) implies that if $g\in G$ intertwines the pair $(E^\times\cdot U^{r_0/2}_G,\widetilde{\te})$, then $g\in E^\times\cdot U^{r_0/2}_G$, whence the irreducibility assertion follows from Theorem \ref{t:supercuspidals-GL-n}.

\mbr

Let us now assume that $r_0\geq 3$ is odd and prove Theorem \ref{t:GL-n}(b). To construct the required representation $\sg$ of $E^\times\cdot U^{(r_0-1)/2}_G$ we use Theorem \ref{t:main-geometric-theorem}; in particular, from now on the notations and conventions of \S\ref{ss:geometric-ingredient} will be in force. The set
\[
J=1+P_E^{r_0-1}+(C\cap P_A^{(r_0-1)/2}) = 1+P_E^{r_0-1}+\sum_{1\neq\ga\in\sG} P_E^{(r_0-1)/2}\cdot\ga
\]
is a compact open subgroup of $G$ which is normalized by $E^\times$. Form the corresponding semidirect product $E^\times\ltimes J$. The multiplication map $E^\times\ltimes J\rar{}E^\times\cdot U^{(r_0-1)/2}_G$ is a surjective group homomorphism and its kernel consists of elements of the form $(h,h^{-1})$, where $h\in U^{r_0-1}_E$. Next consider the open normal subgroup
\[
J_+=1+P_E^{r_0}+(C\cap P_A^{(r_0+1)/2}) = 1+P_E^{r_0}+\sum_{1\neq\ga\in\sG} P_E^{(r_0+1)/2}\cdot\ga \subset J.
\]
The quotient $J/J_+$ can be identified with the group $U(\fqn)$ constructed in \S\ref{ss:geometric-ingredient} via
\[
1+\varpi^{r_0-1}b+\varpi^{(r_0-1)/2}\sum_{\ga\neq 1}a_\ga\cdot\ga \longmapsto 1+\sum_{i=1}^{n-1}\bar{a}_{\vp^i}e_i+\bar{b}e_n,
\]
where $\vp\in\sG$ is the Frobenius, $b,a_\ga\in\cO_E$ and $\bar{a}$ denotes the image of an element $a\in\cO_E$ in the quotient $\cO_E/P_E=\fqn$. The uniformizer $\varpi$ defines an isomorphism $E^\times\rar{}\bZ\times\cO_E^\times$, which yields a natural surjection $E^\times\rar{}\cO_E^\times/U^1_E=\fqn^\times$. This map, together with the map $J\rar{}J/J_+\rar{\simeq}U(\fqn)$ constructed above, yield a surjective homomorphism $E^\times\ltimes J\rar{}\bG_m(\fqn)\ltimes U(\fqn)\rar{\simeq}\cR^\times(\fqn)$.

\mbr

The restriction of $\te$ to $U^{r_0-1}_E$ induces a character $\psi$ of $U^{r_0-1}_E/U^{r_0}_E\cong\fqn$ with trivial $\sG$-stabilizer. Let $\sg_0$ be the pullback to $E^\times\ltimes J$ of the representation of $\cR^\times(\fqn)$ on $H^{n-1}_c(X\tens_{\fqn}\bfq,\ql)[\psi]$ considered in Theorem \ref{t:main-geometric-theorem}. We can also view $\te$ as a character of $E^\times\ltimes J$ via projection onto the first factor, and by construction, the representation $\te\tens\sg_0$ of $E^\times\ltimes J$ is trivial on the subgroup consisting of elements of the form $(h,h^{-1})$, where $h\in U^{r_0-1}_E$. Therefore $\te\tens\sg_0$ descends to an irreducible representation $\sg$ of $E^\times\cdot U^{(r_0-1)/2}_G$. By Theorem \ref{t:main-geometric-theorem}(c), we have $\tr\sg(x)=(-1)^{(n-1)}\cdot\te(x)$ for any very regular $x\in\cO_E^\times$. By construction, the restriction of $\sg$ to $F^\times\cdot U^1_E\cdot U^{(r_0+1)/2}_G$ is a direct sum of copies of a character that equals $\te$ on $F^\times\cdot U^1_E$ and is trivial on $1+(C\cap P^{(r_0+1)/2}_A)$. In particular, with the notation of \S\ref{ss:intertwiners}, the restriction of $\sg$ to $U^{(r_0+1)/2}_G$ is a direct sum of copies of $\te_A$. So if $g\in G$ intertwines $(E^\times\cdot U^{(r_0-1)/2}_G,\sg)$, then $g$ intertwines $(U^{(r_0+1)/2}_G,\te_A)$, and therefore $g\in E^\times\cdot U^{(r_0-1)/2}_G$ by Lemma \ref{l:intertwiners}(a). Hence the representation $\pi=\Ind_{E^\times\cdot U^{(r_0-1)/2}_G}^G(\sg)$ of $G$ is irreducible and supercuspidal by Theorem \ref{t:supercuspidals-GL-n}, and the proof of Theorem \ref{t:GL-n} is complete.

\subsection{Proof of Theorem \ref{t:division-algebra}}\label{ss:proof-t:division-algebra} We saw in \S\ref{ss:preliminaries-main} that it remains to establish parts (a) and (b) of the theorem. The first assertion of (a) is straightforward. To prove the second one, observe that if $r_0$ is odd, then with the notation of \S\ref{ss:intertwiners}, we have $m=s$ and $\widetilde{\te}\bigl\lvert_{U^m_D}=\te_D$. So Lemma \ref{l:intertwiners}(b) implies that if $g\in G$ intertwines the pair $(E^\times\cdot U^m_D,\widetilde{\te})$, then $g\in E^\times\cdot U^{m-1}_D=E^\times\cdot U^m_D$ (we used the fact that $n$ divides $m-1$ in this case), whence the irreducibility assertion follows from Proposition \ref{p:division-algebra-auxiliary}.

\mbr

The proof of Theorem \ref{t:division-algebra}(b) is essentially the same as that of Theorem \ref{t:GL-n}(b). The only essential difference is that now one defines $J=1+P_E^{r_0-1}+(C'\cap P^m_D)$, where $m=\frac{n(r_0-2)}{2}+1$, and $J_+=1+P_E^{r_0}+(C'\cap P^{m+n}_D)$, and identifies $J/J_+$ with $U(\fqn)$ via $1+\varpi^{r_0-1}b+\varpi^{(r_0-2)/2}\sum_{j=1}^{n-1}a_j\cdot\Pi^j \longmapsto 1+\sum_{i=1}^{n-1}\bar{a}_je_j+\bar{b}e_n$. \qed

\bibliographystyle{alpha}
\bibliography{SpecialCasesBibliography}

\end{document}